\documentclass[11pt]{article}
\usepackage{bbm}
\usepackage[top=1in, bottom=1.5in]{geometry}
\usepackage{color}
\usepackage{amsfonts}
\usepackage{latexsym, amssymb, amsmath, amscd, amsthm, amsxtra}
\usepackage{mathtools}
\usepackage{enumerate}
\usepackage[all]{xy}
\usepackage{mathrsfs}
\usepackage{fancyhdr}
\usepackage{listings}
\usepackage{hyperref}

\usepackage{algorithm}
\usepackage{algpseudocode}
\usepackage{framed}
\usepackage{lipsum}

\hypersetup{
	colorlinks   = true, 
	urlcolor     = blue, 
	linkcolor    = blue, 
	citecolor   = red 
}

\newtheorem{theorem}{Theorem}[section]
\newtheorem{lemma}[theorem]{Lemma}

\newtheorem{proposition}[theorem]{Proposition}

\newtheorem{conjecture}[theorem]{Conjecture}

\theoremstyle{definition}
\newtheorem{definition}[theorem]{Definition}

\usepackage{times}

\newcommand{\dif}{\operatorname{d}\!}

\title{
A Proof of The Changepoint Detection Threshold Conjecture in Preferential Attachment Models}

\author{Hang Du\\{MIT} \and Shuyang Gong\\{ Peking University} \and Jiaming Xu\\ Duke University}
\date{\today}

\begin{document}
\maketitle

\begin{abstract}
    We investigate the problem of detecting and estimating a changepoint in the attachment function of a network evolving according to a preferential attachment model on  $n$ vertices, using only a single final snapshot of the network. Bet et al.~\cite{bet2023detecting} show that a simple test based on thresholding the number of vertices with minimum degrees can detect the changepoint when the change occurs at time $n-\Omega(\sqrt{n})$. They further make the striking conjecture that detection becomes impossible for any test if the change occurs at time $n-o(\sqrt{n}).$  Kaddouri et al.~\cite{kaddouri2024impossibility} make a step forward by proving the detection is impossible if the change occurs at time $n-o(n^{1/3}).$ In this paper, we resolve the conjecture affirmatively,  proving that detection is indeed impossible if the change occurs at time $n-o(\sqrt{n}).$ 
Furthermore, we establish that estimating the changepoint with an error smaller than $o(\sqrt{n})$ is also impossible, thereby confirming that the estimator proposed in Bhamidi et al.~\cite{bhamidi2018change} is order-optimal. \footnote{Accepted for presentation at the Conference on Learning Theory (COLT) 2025}
\end{abstract}

\section{Introduction}
\label{sec-intro}

This paper addresses the problem of detecting and estimating a changepoint in the underlying growth dynamics of a network based solely on its final snapshot. Specifically, suppose a network evolves according to the following preferential attachment model~\cite{Hofstad_2016}. It starts with an initial graph $G_2$ consisting of two vertices connected by $m$ parallel edges. At every subsequent time step $3 \le t \le n$,  a new vertex arrives and forms $m$ new incident edges. These edges connect to the pre-existing vertices independently with probabilities proportional to their \emph{current degrees} plus an additive shift $\delta(t)$. Under this formalism, incoming vertices are more likely to attach to vertices that already have a large degree, further increasing the degree of these vertices, thereby capturing ``the rich-get-richer'' phenomenon commonly observed in real-world networks.
The introduction of parameter $\delta(t)$ enables interpolation between the celebrated Barab\'{a}si-Albert model~\cite{barabasi1999emergence} ($\delta(t) \equiv 0$) with the uniform attachment model ($\delta(t) \equiv \infty$). It is evident that the  smaller $\delta_t$, the stronger preference for high-degree vertices.

In real-world networks, the extent of degree preference, captured by $\delta(t)$,  may change due to the occurrence of some major event or the implementation of some new policy or intervention~\cite{cirkovic2022likelihood}. Thus, a natural yet fundamental question is whether such a change can be detected from  network data. As such, we consider the following hypothesis testing problem where under the null hypothesis $\mathcal{H}_0$, $\delta(t) \equiv \delta$ remains constant, while under the alternative hypothesis $\mathcal{H}_1$, $\delta(t)$ changes from $\delta$ to $\delta'$ at an unknown time $\tau_n$, referred to as the changepoint. The goal is to test these hypotheses based solely on the final snapshot of the network, $G_n$. See~Fig.~\ref{fig:PA_detection} for a graphical illustration. Another closely related statistical task is to estimate the changepoint time $\tau_n$ under $\mathcal{H}_1$. Crucially, we do not observe the sequence of graphs $G_t$ generated during $2 \le t \le n-1$. In particular, the arrival times of vertices in $G_n$ are unknown. This setup captures practical yet challenging scenarios where some significant events or interventions may alter the network's evolution, but only the final snapshot of the network is observed. 

This hypothesis testing problem was first formulated and studied in~\cite{bet2023detecting}. Intuitively,  detecting the changepoint becomes increasingly difficult as $\tau_n$ gets larger because the change impacts a smaller portion of the network.\footnote{This can be formalized through an application of the data-processing inequality.
Specifically, let $\mathbb{Q}_{n,t}$ denote the distribution of $G_n$ with the changepoint at time $t$. 
Let $\tau_n \le t_n$ and consider the transformation  that regenerates the network from $\tau_n$ to $t_n$ with $\delta(t) =\delta$. Then this transformation maps distribution $\mathbb{Q}_{n,n}$ to itself and $\mathbb{Q}_{n,\tau_n}$ to $\mathbb{Q}_{n,t_n}$. Applying the data-processing inequality (see e.g.~\cite[Theorem 7.4]{polyanskiy2025information}), it follows that
$\mathsf{TV}(\mathbb{Q}_{n,n}, \mathbb{Q}_{n,t_n}) \le\mathsf{TV}(\mathbb{Q}_{n,n}, \mathbb{Q}_{n,\tau_n}). $}  
In other words, we would like to detect the change as quickly as possible~\cite{veeravalli2014quickest,xie2021sequential}. Accordingly, \cite{bet2023detecting} focuses on the \emph{late-change regime}, assuming $\tau_n=n-\Delta$ with $\Delta=c n^\gamma$, where  $c>0$ and $\gamma \in (0,1).$ The authors proposed a minimum-degree test, based on thresholding the number of vertices with minimum degree $m$ in the observed network. They prove that the mean number of degree-$m$ vertices under $\mathcal{H}_0$ and $\mathcal{H}_1$ differ by 
$\Theta(n^\gamma)$, while fluctuations under both hypotheses are $O(\sqrt{n})$.  As a result,  the minimum-degree test achieves \emph{strong detection} (with vanishing Type-I and Type-II errors as $n \to \infty$) if $\gamma >1/2$ and \emph{weak detection} (strictly better than random guessing, with the sum of Type-I and Type-II errors bounded away from $1$ as $n \to \infty$) if $\gamma = 1/2.$  Remarkably, the authors further conjecture that all tests are powerless and fail in weak detection when $\gamma<1/2$. 
\begin{conjecture}~\cite[Conjecture 3.2]{bet2023detecting}
If $\gamma<1/2$, then all tests based on $G_n$ are powerless, that is, the sum of Type-I and Type-II errors converges to $1$ as $n \to \infty.$    
\end{conjecture}
This conjecture is rather striking, because, if true, it implies that neither degree information nor any higher-level graph structure is useful for detection when $\gamma<1/2$. Recently, \cite{kaddouri2024impossibility} made significant progress toward resolving the conjecture by proving the impossibility of strong detection when $\Delta=o(n^{1/3})$ for $\delta>0$  or $\Delta=o(n^{1/3}/\log n)$ for $\delta=0$.
 However, their techniques do not extend to the regime where $\Delta=\Omega(n^{1/3})$ or $\delta<0.$ 
 Moreover, their results do not rule out the possibility of weak detection.

In this paper, we completely resolve the conjecture by proving that weak detection is impossible when $\Delta=o(n^{1/2})$. In fact, our proof implies a much stronger statement, showing that weak detection remains impossible even if, in addition to $G_n,$ the entire network history $G_t$ were observed up to time $t$ such that $\Delta^2 \ll n-t \ll n$. Our impossibility of detection further implies that no estimator can pinpoint the changepoint time $\tau_n$ with an additive error of $o(\sqrt{n})$ with $\Omega(1)$ probability. This shows that the estimator $\widetilde{\tau}_n$ proposed in~\cite{bhamidi2018change}, which achieves $|\widetilde{\tau}_n-\tau_n|= O(\sqrt n)$ with high probability, is order-optimal.\\

\subsection{Further related work}
The preferential attachment (PA) model and its generalization are arguably the most widely adopted randomly growing graph models. Since its introduction in~\cite{barabasi1999emergence}, the PA model has received a tremendous amount of attention thanks to its simplicity of the local connection rules and its explanatory power of the power-law degree distribution commonly observed in the real-world networks. Many structural properties of the PA model, ranging from asymptotic degree distributions to local network structures, have been well-understood by now. We refer the interested readers to the excellent books~\cite{Hofstad_2016,Hofstad_2024} for details.

Changepoint detection in PA models initially appeared in~\cite{bhamidi2018change,banerjee2023fluctuation}, specifically for PA trees (with $m=1$), focusing on the \emph{early-change} regime~\cite{bhamidi2018change,banerjee2023fluctuation}, where $\tau_n$ scales as $\gamma n$ or $n^\gamma$ for some constant $0<\gamma<1$. Through embedding the PA tree model into a continuous time branching process, consistent estimation of the changepoint is shown to be achievable based on either the asymptotic degree distribution or the maximum degree.   Subsequent work~\cite{cirkovic2022likelihood} proposes a likelihood-ratio based procedure to detect and estimate the changepoint, and establishes its consistency and asymptotic distribution. The results are further extended to detect multiple changepoints. However, the methods and results in~\cite{cirkovic2022likelihood} crucially rely on the observation of the entire network history. In comparison, our work, following~\cite{bet2023detecting}, assumes only a snapshot of the final network is observed. Also, our results hold for general preferential attachment models beyond trees with initial degree $m \ge 1$.  

In the absence of any change-point, the problem of estimating the general attachment functions was studied previously in~\cite{Gao2017} under the PA tree models, and consistent estimation is established based on empirical degrees. In the special case of affine attachment functions, the consistency and asymptotic normality of the maximum likelihood estimator are established for trees \cite{gao2021statistical} and random initial degrees~\cite{gao2017asymptotic}.

More broadly, our work contributes to the rapidly evolving field of detection and estimation in preferential attachment models. 
As will be seen later, one main challenge in our analysis is to deal with the unobserved vertex arrival order. A key driving intuition underlying our proof, albeit not precise, is that 
for the last $N$ arriving vertices with $\Delta^2 \ll N \ll n$, the final network snapshot contains some but very little information on their arrival order; Consequently, the identities of the last $\Delta$ vertices are almost ``hidden'' among these last $N$ vertices, so that the change of network structure  is undetectable. 
In this sense, our work nicely connects to a very active research field of \emph{network archeology}, which seeks to infer the vertex arrival order from the final network snapshot~\cite{magner2018times,young2019phase,crane2021inference,briend2024estimating}. 
A special instance, particularly well-studied with numerous deep results, is that of root finding, when one is only interested in estimating the first arriving vertex~\cite{Haigh1970,shah2011rumors,shah2016finding,BubeckMosselRacz2015,CurienEtAl2015,BubeckDevroyeLugosi2017,BubeckEldanMosselRacz2017,DevroyeReddad2019,KhimLoh2016,LugosiPereira2019,AddarioBerry2021,crane2021inference,banerjee2022root,brandenberger2022root,banerjee2023degree,briend2023archaeology,contat2024eve}.  
It is worth noting that the challenge, where only the final network snapshot is observed while the vertex arrival order is hidden, also naturally appears in many other related estimation problems in PA models, such as community detection~\cite{ben2025inference}, correlation detection~\cite{racz2022correlated}, and graph matching~\cite{Korula2014}. 
We believe that the new technical tools—particularly our likelihood ratio bounds and the associated coupling techniques—may prove useful in addressing a range of problems within this domain.

Lastly, changepoint detection and localization have also received extensive attention in dynamic graph models beyond PA models, such as dynamic stochastic block and graphon models, see e.g.~\cite{wang2013locality,zhao2019change,Bhattacharjee2020,wang2021optimal,enikeeva2025change} and references therein. However, in contrast to the current paper, those works often assume networks do not grow in sizes, while edge connections are changing dynamically, and the goal is to detect the change of the underlying edge probability matrix based on the entire network evolution history. 

\subsection{Notation and organization}
We use $G_n = (V_n, E_n)$ to denote the network snapshot at time step $n$, and $\overline{G}_n$ to denote the sequence of networks generated up to time $n$. The changepoint time is given by $\tau_n = n - \Delta$. We use $P$ and $Q$, in different font styles, to distinguish certain distributions under the null hypothesis $\mathcal{H}_0$ and the alternative hypothesis $\mathcal{H}_1$, respectively.
For probability measures $P$ and $Q$, their total variation distance is given by $\operatorname{TV}(P,Q)=\frac{1}{2}\int |dP-dQ|$. We use standard big $O$ notations,
e.g., for any sequences $\{a_n\}$ and $\{b_n\}$, we say that $a_n = O(b_n)$ if there exists an absolute constant $C>0$ such that $|a_n| \leq C |b_n|$ for all $n$ large enough. We say that $a_n = o(b_n)$ if $\lim_{n\to \infty} a(n)/b(n) = 0$. We write $ a_n= \Omega(b_n)$ if $b_n = O(a_n)$, $a_n = \omega(b_n)$ if $b_n=o(a_n)$, and $a_n=\Theta(b_n)$ if $a_n=O(b_n)$ and $b_n=O(a_n).$

The remainder of the paper is structured as follows.  Section~\ref{sec:main_results} formally introduces the problem setup, states our main results, and includes a concise proof that the impossibility of changepoint detection implies the impossibility of localization.   Section~\ref{sec:overview_proofs} provides a high-level overview of the key proof ideas behind Theorem~\ref{thm-main}. In Section~\ref{sec-proof-of-main-thm}, we present the formal proof of Theorem~\ref{thm-main}, relying on a variance bound on  the likelihood ratio  stated in Proposition~\ref{prop-variance}. The most technically involved portion, Section~\ref{sec:bounding_norm}, is devoted entirely to proving Proposition~\ref{prop-variance}. Finally, we conclude with a discussion of open questions and future directions in Section~\ref{sec:discussion}.
\section{Problem setup and main results}\label{sec:main_results}
In this section, we formally introduce the problem setup and our main results. 

\begin{definition}[Preferential attachment model]\label{def-preferential-attachment-model}
	Given $m,n\in \mathbb{N}$, a sequence $\{\delta_t\}_{t=2}^n$ with $\delta_t>-m$ for each $t$, and a set $V_n$ of cardinality $n$, an undirected graph $G_n=(V_n,E_n)$ with the vertex set $V_n$ and edge set $E_n$ is defined as follows:
	\begin{itemize}
		\item The initial graph $G_2=(V_2, E_2)$ consists of two vertices $v_1\neq v_2$ chosen from $V_n$ uniformly at random and $m$ multiple edges connecting them; 
		\item For $3 \le t \le n$, graph $G_t=(V_t, E_t)$ is obtained by adding to $G_{t-1}$ a new vertex $v_t$ chosen from $V_n\setminus V_{t-1}$ uniformly at random and connecting $m$ edges from $v_t$ to vertices in $V_{t-1}$. 
		Specifically, we construct a graph sequence $G_{t,0},G_{t,1},\ldots,G_{t,m}$ starting from $G_{t,0}=G_{t-1}$ and ending at $G_{t,m}=G_t$. For $1\leq i \leq m$, the graph $G_{t,i}$ is obtained by adding an edge from $v_t$ to a vertex $v_{t,i} \in V_{t-1}$ with conditional probability 
		\begin{align}\label{eq-prob-def}
			\mathbb{P} \left[ v_{t,i}=v \mid G_{t,i-1} \right] = \frac{\mathsf{deg}_{G_{t,i-1}}(v)+\delta_t}{ \sum_{u \in V_{t-1} } \left( \mathsf{deg}_{G_{t,i-1}}(u) +\delta_t \right) }, \quad \forall v \in V_{t-1},
		\end{align}
		where $\mathsf{deg}_{G_{t,i-1}}(v)$ is the degree of vertex $v$ in $G_{t,i-1}$. 
	\end{itemize}
\end{definition}

We abbreviate the entire network history $\{G_{t,i}\}_{3\leq t\leq n,1\leq i\leq m}$ as $\overline{G}_n$ and denote its distribution by 
$\mathcal{P}_{m,n,\{\delta_t\}}$. 
The marginal distribution of the final network snapshot $G_n$ is denoted by $\mathbb P_{m,n,\{\delta_t\}}.$

\begin{definition}[Detection problem] 
	The changepoint detection problem with parameters $(n,m,\delta, \delta', \tau_n)$ refers to the following problem of distinguishing hypotheses:
	\begin{align*}
		& \mathcal{H}_0: \quad G \sim  \mathbb{P}_{m,n,\delta}  \triangleq \mathbb{P}_{m,n,\{\delta_t\}} \text{ with $\delta_t\equiv \delta $ for $2 \le t \le n$\,, } \\
		& \mathcal{H}_1: \quad G \sim  \mathbb{Q}_{m,n,\delta,\delta',\tau_n}\triangleq \mathbb P_{m,n,\{\delta_t\} } \text{ with $\delta_t\equiv \delta\mathbf{1}\{t\le \tau_n\}+\delta'\mathbf{1}\{t>\tau_n\}$\,.  }
	\end{align*}
\end{definition}

\begin{figure}[h]
\centering
\includegraphics[width=0.8\linewidth]{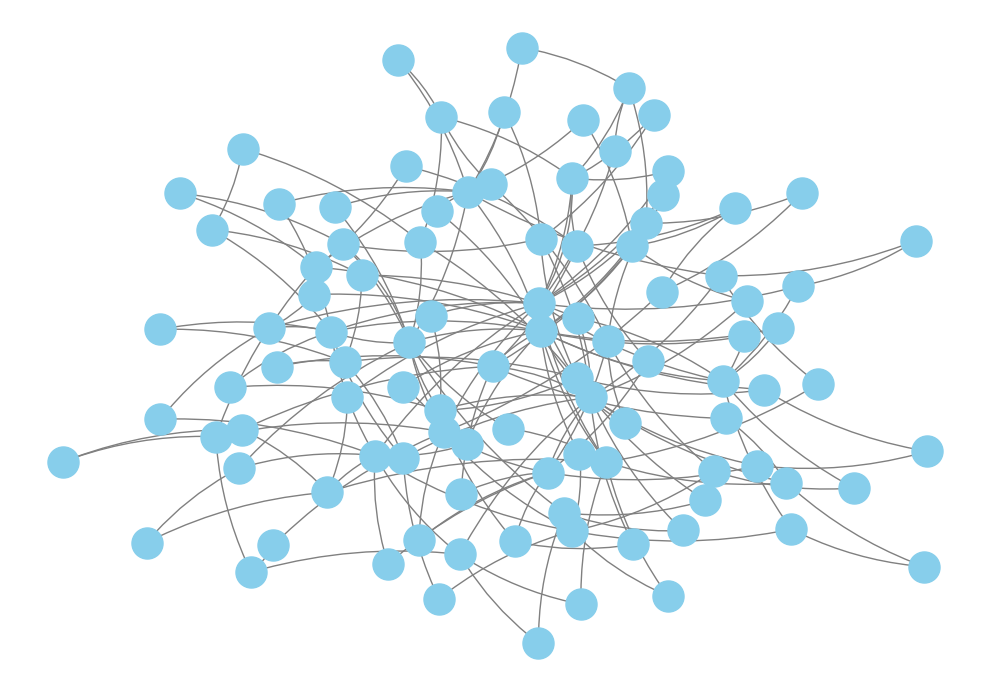}
\caption{A final network snapshot of a preferential attachment graph with $m=2$ and $n=100$. The goal is to detect whether there is a change in $\delta(t)$.
}
\label{fig:PA_detection}
\end{figure}

We focus on the regime where $m \in \mathbb{N}$ and $\delta \neq \delta'$ are fixed constants while the network size $n \to \infty$.  
For ease of notation, we abbreviate $\mathbb{P}_{m,n,\delta}  $ as $\mathbb{P}_n$
and 
$\mathbb{Q}_{m,n,\delta,\delta',\tau_n}$ as $\mathbb{Q}_{n,\tau_n}$, where the first subscript indicates the network size and the second subscript indicates the changepoint time. 
Then $\mathbb{P}_{n}$ is equivalent to $\mathbb{Q}_{n,n}$, as no changepoint is equivalent to the changepoint being at the last time step $n.$

In \cite{bet2023detecting}, the authors introduced a testing statistic based on the number of vertices with minimum degree $m$. They proved that this statistic achieves strong detection between $\mathbb{Q}_{n,n}$ and $\mathbb{Q}_{n,\tau_n}$, provided that $n-\tau_n=\omega(\sqrt n)$ and weak detection when $n-\tau_n=\Omega(\sqrt n)$. Our first result shows that these bounds are essentially optimal. 


%

\begin{theorem}\label{thm-main}
	If $\tau_n$ satisfies $n-\tau_n=o(\sqrt n)$, then $\operatorname{TV}(\mathbb{Q}_{n,n},\mathbb{Q}_{n,\tau_n}) =o(1)$, where $\operatorname{TV}$ stands for the total variation distance. 
\end{theorem}
Since the minimum sum of Type-I and II errors is equal to one minus the TV distance,  it follows from Theorem~\ref{thm-main} that all tests are asymptotically powerless, that is,  their sum of the Type-I and Type-II errors converges to $1$ as $n \to \infty.$ This resolves~\cite[Conjecture 3.2]{bet2023detecting} in positive.

Regarding the task of estimating changepoint, \cite[Theorem 2.4]{bhamidi2018change} shows that assuming $\tau_n \ge \varepsilon n$ for some constant $\varepsilon \in (0,1)$, there exists an estimator $\widetilde\tau_n$ based on $G_n$ such that $|\widetilde{\tau}_n-\tau_n|= O(\sqrt n)$ with high probability.\footnote{While the estimator defined in~\cite[Equation (2.18)]{bhamidi2018change} requires the observation of the graph sequence $\{G_t\}$ for $\epsilon n \le t \le n$, it suffices to estimate $\tau_n$ based on the fraction of degree-$m$ vertices in the final network $G_n$. In particular, by \cite[Theorem~2.3]{bhamidi2018change}, the fraction of degree-$m$ vertices in $G_n$ approaches the asymptotic limit $p_1^\infty$ within $O(n^{-1/2})$ error. Therefore, by inverting the expression of $p_1^\infty$ given in \cite[Equation (2.3)]{bhamidi2018change} which depends on $\tau_n/n$, one can pinpoint $\tau_n/n$ within $O(n^{-1/2})$ error.}
As an immediate consequence of Theorem~\ref{thm-recovery}, we prove that this result is order-optimal. 

\begin{theorem}\label{thm-recovery}
	Assume that there exists $\varepsilon\in (0,1)$ such that $\tau_n\ge \varepsilon n$. Then there is no estimator $\widehat{\tau}_n$ based on $G_n\sim \mathbb{Q}_{n,\tau_n}$ such that $|\widehat{\tau}_n-\tau_n|=o(\sqrt n)$ holds with non-vanishing probability uniformly for all $ \tau_n\in [\varepsilon n,n] $. 
\end{theorem}


\begin{proof}[Proof of Theorem~\ref{thm-recovery} assuming Theorem~\ref{thm-main}]
	We first show that 
	$$
    \operatorname{TV}(\mathbb{Q}_{n,\tau_n},\mathbb{Q}_{n,\sigma_n})=o(1)
	$$
	for all $\sigma_n$ such that $|\sigma_n-\tau_n|=o(\sqrt n).$
	Without loss of generality, assume $\sigma_n\le \tau_n$. Consider the transformation that grows from $G_{\tau_n}$ to $G_n$ according to the preferential attachment model with $\delta(t)=\delta'$ for $\tau_n \le t \le n.$ If the input graph $G_{\tau_n} \sim \mathbb{Q}_{\tau_n,\tau_n} $
    then
	the output graph $G_n \sim  \mathbb{Q}_{n,\tau_n}$. 
    Analogously, if the input graph 
    $G_{\tau_n} \sim \mathbb{Q}_{\tau_n,\sigma_n}$, then
	the output graph  $G_n \sim \mathbb{Q}_{n,\sigma_n}.$
	By the data-processing inequality (see e.g.~\cite[Theorem 7.4]{polyanskiy2025information}), the TV distance does not increase after the transformation. Therefore, 
	\begin{align}
\operatorname{TV}\left(\mathbb{Q}_{n,\tau_n},\mathbb{Q}_{n,\sigma_n}\right)
		\le \operatorname{TV}\left( \mathbb{Q}_{\tau_n,\tau_n},\mathbb{Q}_{\tau_n,\sigma_n}\right) \le o(1), \label{eq:dp_1}
	\end{align}
	where the last inequality follows from applying Theorem~\ref{thm-main} with $n$ substituted by $\tau_n \ge \varepsilon n$, 
    $\tau_n$ substituted by $\sigma_n$, and using the assumption $\sigma_n \ge \tau_n-o(\sqrt{n}).$
	
	
	
	Next, we prove the statement of Theorem~\ref{thm-recovery} via contradiction. Assume that there exists a constant $c>0$ as well as a sequence $\gamma_n\to 0$ such that for any $n$, there exists $\widehat{\tau}_n\equiv\widehat{\tau}_n(G_n)$ satisfying
	\[
	\mathbb{Q}_{n,\tau_n}[|\widehat{\tau}_n-\tau_n|\le \gamma_n \sqrt n]\ge c
	\]
	uniformly for all $\tau_n \in [\varepsilon n,n]$. Pick 
    $$
    \tau_n^k=\varepsilon n+3k\gamma_n\sqrt n, \quad \text{ for } k=0, 1,\dots,L\triangleq \lfloor\gamma_n^{-1/2}\rfloor.
    $$
    Since $\max_{k,\ell}|\tau_n^k-\tau_n^\ell|\le 3\sqrt{\gamma_n}\cdot \sqrt{n}=o(\sqrt n)$, it follows from~\eqref{eq:dp_1} that   
    $$
    \operatorname{TV}(\mathbb{Q}_{n,\tau_n^k},\mathbb{Q}_{n,\tau_n^\ell})=o(1), \quad \forall 
    0\le k,\ell\le L.
    $$ 
    This implies 
	\[
\mathbb{Q}_{n,\tau_n^{0}}\big[|\hat{\tau}_n-\tau_n^k|\le \gamma_n\sqrt n\big]
\ge \mathbb{Q}_{n,\tau_n^k}\big[|\hat{\tau}_n-\tau_n^k|\le \gamma_n\sqrt n\big]-o(1)\ge c-o(1)\,.
	\]
	However, it is clear that the events $\mathcal{E}_k\triangleq \{|\hat{\tau}_n -\tau_n^k|\le \gamma_n\sqrt n\}$ are mutually disjoint for  $k=0,1,\dots,L$. Therefore, we conclude that 
    $$
   1  \ge \mathbb{Q}_{n,\tau_n^{0}}\big[
   \cup_{k=0}^L \mathcal{E}_k \big] 
   = \sum_{k=0}^L \mathbb{Q}_{n,\tau_n^{0}}\big[ \mathcal{E}_k \big]
     \ge (L+1)(c-o(1))
     \ge \gamma_n^{-1/2}(c-o(1)),
    $$
    which contradicts to the assumption that $\gamma_n\to 0$. This completes the proof.    
\end{proof}


\section{Overview of proof ideas for Theorem~\ref{thm-main}} \label{sec:overview_proofs}
Before presenting the formal proof, we briefly outline the main challenges, the limitations of the proof techniques from~\cite{kaddouri2024impossibility}, and our new ideas. For ease of exposition, we focus on the tree case $m=1$ and refer to the pre-existing vertex that an arriving vertex $v$ attaches to as the \emph{parent} of $v$. The main ideas can be readily extended to general $m \ge 1.$

A celebrated approach to prove the impossibility of detection in high-dimensional statistics and network inference is to bound the second-moment of the likelihood ratio between the alternative and null distributions, see e.g.~\cite{wu2018statistical} for a survey of this technique. However, the likelihood ratio between the alterative distribution $\mathbb{Q}_{n,\tau_n}$ and null distribution 
$\mathbb{P}_{n}$ involves an average over all admissible vertex arrival times, rendering its second moment too complex to be bounded directly. One natural idea is to simplify the likelihood ratio by revealing all vertex arrival times, i.e., the entire network history. However, as shown in~\cite{kaddouri2024impossibility}, this overly simplifies the detection problem, with the (strong) detection threshold becoming  $\Delta \to \infty$. To address this, \cite{kaddouri2024impossibility} reveals the arrival times of all vertices except for a carefully chosen subset of leaf vertices, denoted by $\mathcal S$ (highlighted in red and bold in Fig.~\ref{fig:PA_easy_vis}). Despite this refinement, the approach still reveals too much information, resulting in the proof that detection is impossible only for $\Delta=o(n^{1/3}).$  In more detail, 
given a time threshold $\tau_n'=n-\Delta'$, $\mathcal S$ includes all leaf vertices $v$ such that (1) the parent of $v$ arrives no later than $\tau_n'$; (2) and $v$ is the only child of its parent that arrives later than $\tau_n'.$ 
It can be shown that $|\mathcal S| \approx \Delta'$, and furthermore, $\mathcal S$
contains all vertices arriving after $\tau_n$ with probability approximately $(1-\Delta'/n)^\Delta$, which is $1-o(1)$ when $\Delta' \Delta=o(n).$ Since the arrival times of all vertices in $\mathcal S$ are  interchangeable, detection is expected to be impossible if $|\mathcal S| \asymp \Delta' \gg \Delta^2$. This holds when $\Delta \ll n^{1/3}$ by choosing $\Delta^2 \ll \Delta' \ll n/\Delta$.

\begin{figure}[h]
\centering
\includegraphics[width=0.8\linewidth]{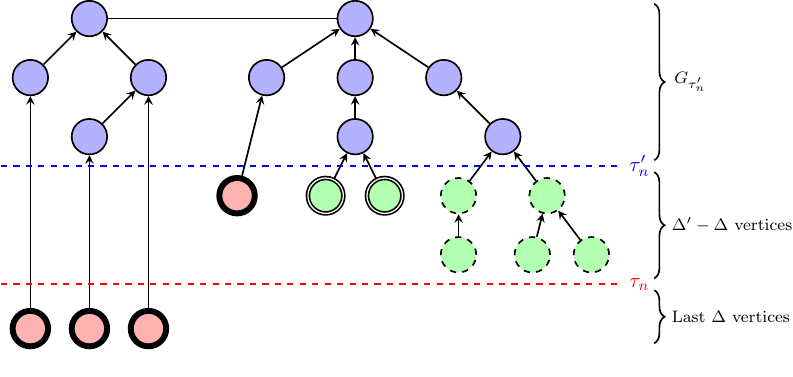}
\caption{Figure credit~\cite{kaddouri2024impossibility}: Typical Preferential attachment graph with $m=1$ and $\Delta=o(n^{1/3}).$ The arrows are pointing from a vertex to its parent. The arrival times of all vertices, except for the set $\mathcal{S}$ of leaf vertices (highlighted in bolded red), are revealed. 
}
\label{fig:PA_easy_vis}
\end{figure}

\begin{figure}[h]
\centering
\includegraphics[width=0.8\linewidth]{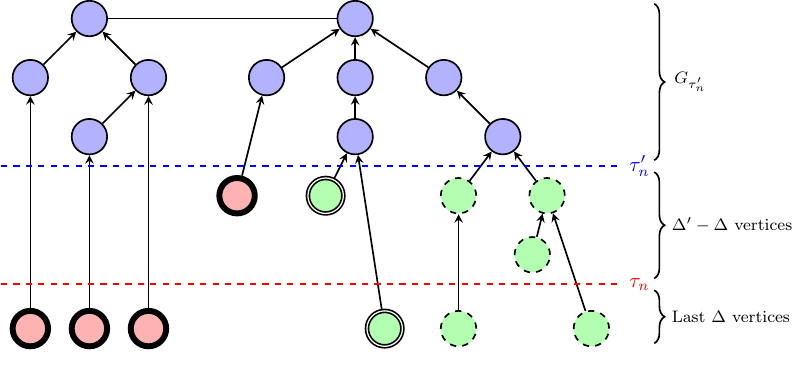}
\caption{Figure credit~\cite{kaddouri2024impossibility}: Typical Preferential attachment graph with $m=1$ and $\Delta=\Omega(n^{1/3}).$ 
The arrows are pointing from a vertex to its parent.
Vertices arriving after $\tau_n$ may attach to vertices arrived earlier in $[\tau_n'+1,\tau]$, as shown by dashed vertices at bottom.}
\label{fig:PA_hard_vis}
\end{figure}

To extend the impossibility result to $\Delta=o(n^{1/2})$, one can only reveal the vertices  arriving up to time $n-\Delta'$, where $\Delta^2 \ll  \Delta' \ll n$. However, when $\Delta=\Omega(n^{1/3})$ so that $\Delta'\gg n^{2/3}$, some vertices arriving after $\tau_n$ may attach to vertices arrived in $[\tau_n'+1, \tau_n]$, as illustrated in~Fig.~\ref{fig:PA_hard_vis}. As a consequence, the second moment of the likelihood ratio is still too complex to be directly bounded. Thus, some fresh new ideas are needed. Our proof proceeds in four steps.

\paragraph*{Step 1: Interpolation} Our starting point is as follows. As illustrated in Fig.~\ref{fig:PA_hard_vis}, vertices arriving after $\tau_n'=n-\Delta'$ form a subgraph consisting of vertex-disjoint connected components. Although vertices arriving after $\tau_n$ may attach to vertices arrived in $[\tau_n'+1, \tau_n]$, each of them typically attaches to a distinct component or forms a component by itself. Consequently, these vertices behave approximately ``independently'' of each other. This suggests that we may tightly bound the total variation distance $\operatorname{TV}\left(\mathbb{P}_n,\mathbb{Q}_{n,\tau_n} \right)$ by applying the triangle inequality across the changepoint time from $n-\Delta$ to $n-1$. More formally, recall that $\mathbb{Q}_{n,n-k}$ denotes the distribution of the final network snapshot $G_n$, where the first index indicates the size of the network and the second index indicates the time of the changepoint. Thus, we arrive at a sequence of distributions interpolating between 
the null distribution $\mathbb{P}_n=\mathbb{Q}_{n,n}$ and the alterative distribution $\mathbb{Q}_{n,\tau_n}=\mathbb{Q}_{n,n-\Delta}$:
 $$  \mathbb{P}_n=\mathbb{Q}_{n,n} \to \mathbb{Q}_{n,n-1} \to \mathbb{Q}_{n,n-2} \to \cdots \to \mathbb{Q}_{n,n-\Delta-1} \to \mathbb{Q}_{n,n-\Delta} = \mathbb{Q}_{n,\tau_n}.$$
Applying the triangle inequality across the interpolation path, we get that 
\begin{align}
\operatorname{TV}\left(\mathbb{Q}_{n,n},\mathbb{Q}_{n,\tau_n} \right)
& \le\sum_{k=1}^\Delta\operatorname{TV}\left(\mathbb{Q}_{n,n-k+1}, \mathbb{Q}_{n,n-k}\right) \\\nonumber 
&\le \sum_{k=1}^\Delta\operatorname{TV}\left(\mathbb{Q}_{n-k+1,n-k+1}, \mathbb{Q}_{n-k+1,n-k}\right) \nonumber  \\
& = \sum_{k=1}^\Delta\operatorname{TV}\left(\mathbb{P}_{n-k+1}, \mathbb{Q}_{n-k+1,n-k}\right), \nonumber 
\end{align}
where the second inequality holds by applying the data-processing inequality (see e.g.~\cite[Theorem 7.4]{polyanskiy2025information}), since observing the network snapshot at a later time cannot increase the TV distance. 
Thus, we have successfully reducing the task to showing $\operatorname{TV}\left(\mathbb{P}_{n'},\mathbb{Q}_{n',n'-1} \right)=o(1/\Delta)$ for all $n' \in [n-\Delta+1, n].$ Since $n'=n-o(n)$, without loss of generality,  henceforth we can assume $n'=n$ and $\tau_n=n-1$, i.e., the changepoint happens one step before the final time. 

\paragraph*{Step 2: Consider an ``easier model''} To further simplify the likelihood ratio between $\mathbb{Q}_{n,n-1}$ and $\mathbb{P}_{n}$, we  reveal the network history up to time $M=n-N$, denoted by $\overline{G}_M.$ where $\Delta^2 \ll N \ll n.$ Let $\mathcal{P}$ and $\mathcal{Q}$ denote the joint law of $\overline{G}_M$ and $G_n$, under $\mathcal{H}_0$ and $\mathcal{H}_1$, respectively. By the data-processing inequality and Jensen's inequality, we can show 
   \begin{align*}
    \mathrm{TV}(\mathbb{P}_{n},\mathbb{Q}_{n,n-1})  \le  \mathbb{E}_{\overline{G}_M \sim \mathcal{P}_{\overline{G}_M}}  \left[  \mathrm{TV}\left(\mathcal{P}_{G_n \mid \overline{G}_M},\mathcal{Q}_{G_n \mid \overline{G}_M} \right) \right],
    \end{align*}
    where $\mathcal{P}_{G_n \mid \overline{G}_M}$ and $\mathcal{Q}_{G_n \mid \overline{G}_M}$ denote the distribution of $G_n$ conditional on $\overline{G}_M$,  under $\mathcal{H}_0$ and $\mathcal{H}_1$, respectively.
Thus, we further reduce the task to proving 
$\mathrm{TV}\left(\mathcal{P}_{G_n \mid \overline{G}_M},\mathcal{Q}_{G_n \mid \overline{G}_M} \right)=o(1/\Delta)$. 

\paragraph*{Step 3: Bound the second moment} 
Define  the likelihood ratio 
    $
    \mathcal{L} \triangleq \frac{\mathcal{Q}_{G_n \mid \overline{G}_M}}{\mathcal{P}_{G_n \mid \overline{G}_M}}. 
    $
    Then 
    \begin{align*}
    2\mathrm{TV}\left(\mathcal{P}_{G_n \mid \overline{G}_M},\mathcal{Q}_{G_n \mid \overline{G}_M} \right) & = \mathbb{E}_{\mathcal{P}_{G_n \mid \overline{G}_M} } \left[ |\mathcal{L}-1 | \right] \le \sqrt{ \mathrm{Var}_{\mathcal{P}_{G_n \mid \overline{G}_M}}\left[ \mathcal{L} \right]}.
    \end{align*}  
Thus, it suffices  to show 
 $
 \mathrm{Var}_{\mathcal{P}_{G_n \mid \overline{G}_M}}\left[ \mathcal{L} \right]=O(1/N),
 $
 as $N \gg \Delta^2$. Let $V$ denote the set of vertices arriving after time $M=n-N$. 
Consider the subgraph of $G_n$ induced by $V$, which can be decomposed as a vertex-disjoint union of connected components, as illustrated in Fig.~\ref{fig:PA_hard_component}.

\begin{figure}[h]
\centering
\includegraphics[width=0.8\linewidth]{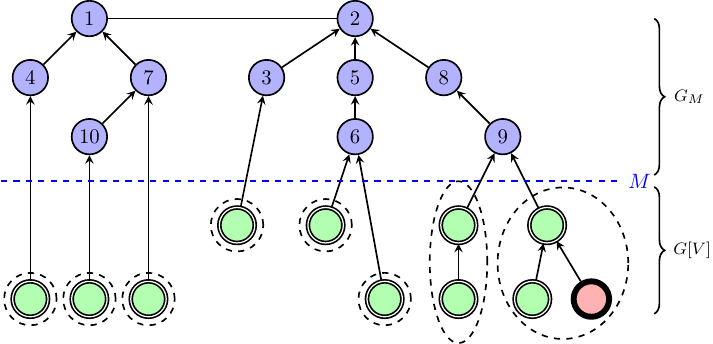}
\caption{Typical Preferential attachment graph with $m=1$, $\Delta=o(\sqrt{n})$, and $\tau_n=n-1.$ The arrows are pointing from a vertex to its parent. The arrival times of vertices arriving up to $M$ have been revealed. The connected components are denoted by dashed ellipses. The bolded red vertex denotes the last-arriving vertex. 
}
\label{fig:PA_hard_component}
\end{figure}

Crucially, the connected components can arrive in any relative order, and only one vertex arrives after the changepoint $\tau_n=n-1$. Thus, 
we can derive a closed-form expression for the likelihood ratio as follows:
$$
\mathcal{L} \triangleq \frac{\mathcal{Q}_{G_n \mid \overline{G}_M}}{\mathcal{P}_{G_n \mid \overline{G}_M}} =  \frac{C_1}{N} \sum_{v \in V } \left| \mathcal{C}(v) \right| \lambda_v X_v,
$$
where $C_1$ is some bounded constant, $\mathcal{C}(v)$ denotes the connected component containing $v \in V$,  $\lambda_v$ denotes the probability that the last-arriving vertex is $v$ conditional on it belongs to $\mathcal{C}(v)$, and $X_v$ denotes the ratio of the attachment probability if there is a change from $\delta'$ to $\delta$ versus no change (cf.~Proposition~\ref{prop-likelihood-ratio-one-step}). Note that both $\lambda_v$ and $X_v$ are bounded above and below by absolute constants. Thus, the variance of $\mathcal{L}$ would be largely driven by the fluctuation of component sizes $|\mathcal{C}(v)|$. Although the sizes of these components vary, we expect that they are typically $\Theta(1)$ and approximately independent, which suggests that $\mathrm{Var}_{\mathcal{P}_{G_n \mid \overline{G}_M}}\left[\mathcal{L} \right]=O(1/N)$. However, rigorously proving this remains the last challenge.

\paragraph*{Step 4. Efron-Stein inequality and coupling}
To bound $\mathrm{Var}_{\mathcal{P}_{G_n \mid \overline{G}_M}}\left[ \mathcal{L} \right]$, we first encode the conditional law $\mathcal{P}_{G_n \mid \overline{G}_M}$ using $N m$ independent random variables $\{U_{t,i}\}_{M<t\le n, 1 \le i \le m}$. This can be easily done in the special case where $m=1$ and $\delta=0,$ where each new arriving vertex connects to an existing vertex $v$ with probability proportional to its degree. Equivalently, $v$ can be chosen by first sampling from all existing edges and then picking one of its two endpoints, uniformly at random. Consequently,
 $ \mathcal{P}_{G_n \mid \overline{G}_M}$  can be encoded by $N$ independent uniform random variables supported over $[2(M-1)], [2M], \ldots, [2(n-2)]$, respectively. This encoding scheme, with a bit of delicate adjustment, can be extended to general $m \ge 1$ and $\delta>-m.$ Now, letting 
     $
     U=(U_{M+1,1}, \ldots, {U_{t,i}}, \ldots, U_{n,m})$ and $U^{(t,i)} = (U_{M+1,1}, \ldots, U'_{t,i}, \ldots, U_{n,m})$, where  $U'_{t,i}$ is an independent copy of $U_{t,i}$,
     we can then write $\mathcal{L}$ as $f(U)$ and apply the Efron-Stein inequality to get that 
     $$
     \mathrm{Var}[\mathcal{L}] \le \frac{1}{2} \sum_{ M<t \le n} \sum_{1 \le i \le m} \mathbb{E}\left[ \left( f(U)- f( U^{(t,i)} ) \right)^2 \right].
     $$
Crucially, our encoding scheme ensures that resampling $U_{t,i}$  can only affect $C(t)$ (the component containing vertex arrived at time $t$) and $C'(t)$ (the counterpart when $U_{t,i}$ is replaced by $U'_{t,i}$), so that  $$
    \left| f(U)- f( U^{(t,i)} ) \right| \le O\left(  \frac{ |\mathcal{C}(t)| + |\mathcal{C}'(t)| }{N} \right). 
     $$
 Finally,  we show the growth of $\mathcal{C}(t)$ can be dominated by a sub-critical branching process to conclude $\mathbb{E}[ |\mathcal{C}(t)|^2] =O(1)$, thereby completing the proof.   

\section{Proof of Theorem~\ref{thm-main}}\label{sec-proof-of-main-thm}



In the following, we present the proof of Theorem~\ref{thm-main}. 
Noticing that $\mathbb{P}_n=\mathbb{Q}_{n,n}$, our proof starts with a simple triangle inequality,
\[
\operatorname{TV}\left(\mathbb{P}_n,\mathbb{Q}_{n,\tau_n} \right)\le\sum_{k=1}^\Delta\operatorname{TV}\left(\mathbb{Q}_{n,n-k+1}, \mathbb{Q}_{n,n-k}\right)\,.
\]
Furthermore, by the data processing inequality (as argued in the Proof of Theorem~\ref{thm-recovery} in Section~\ref{sec-intro}), we have for any $1\le k\le \Delta$, 
$$\operatorname{TV}(\mathbb{Q}_{n,n-k+1}, \mathbb{Q}_{n,n-k})\le 
\operatorname{TV}(\mathbb{Q}_{n-k+1,n-k+1}, \mathbb{Q}_{n-k+1,n-k})
=
\operatorname{TV}(\mathbb{P}_{n-k+1},\mathbb{Q}_{n-k+1,n-k})\,.$$
Therefore, it suffices to show that 
\begin{equation}\label{eq-goal}
	\operatorname{TV}(\mathbb{P}_{n-k+1},\mathbb{Q}_{n-k+1,n-k})=o\left(\frac{1}{\Delta}\right)\,,\forall 1\le k\le \Delta\,.
\end{equation}

We focus on the proof of \eqref{eq-goal} for the case $k=1$, as the proof of the other cases can be obtained by substituting $n$ with $n'\triangleq n-k+1$. 
In what follows we further abbreviate 
$\mathbb{P}$ and $\mathbb{Q}$ for 
$\mathbb{P}_n$ and $\mathbb{Q}_{n,n-1}$,
respectively. 
Our basic strategy is to use the $\chi^2$-norm to control the TV distance. However, a direct computation of the likelihood ratio between $\mathbb{Q}$ and $\mathbb{P}$ leads to a fairly complicated expression, making it difficult to obtain a nice quantitative upper bound on the $\chi^2$-norm as claimed in \eqref{eq-goal}. We circumvent this difficulty by introducing a series of simplified models. We show that on the one hand, the original models are mixtures of the simplified models so it suffices to control the TV distance between simplified models, and on the other hand, the likelihood ratio between the simplified models is much more tractable, enabling us to obtain the desired $\chi^2$-norm upper bound. 

We now proceed to introduce the simplified models. 
We introduce a parameter $N\equiv N(n)$ such that $\Delta^2\ll N\ll n$, and we write $M=n-N$. 
Let $\overline{G}_M$ denote the network history $\{G_{t,i}\}_{3\leq t\leq M,1\leq i\leq m}$ up to time $M.$ 
Let $\mathcal P_{\overline{G}_M, G_n}$ (resp. $\mathcal Q_{\overline{G}_M, G_n}$)
denote the joint distribution of network snapshots $\overline{G}_M$ and $G_n$ under $\mathcal{H}_0$ (resp. $\mathcal{H}_1$). 
Then the marginals satisfy $\mathcal P_{G_n}=\mathbb P$, $\mathcal Q_{G_n}=\mathbb Q$, and 
$\mathcal P_{\overline{G}_M}=\mathcal Q_{\overline{G}_M}$. 
It follows that 
\begin{align}
	\operatorname{TV}(\mathbb P,\mathbb Q) \le \operatorname{TV}\left(\mathcal P_{\overline{G}_M, G_n}\;, \mathcal Q_{\overline{G}_M, G_n} \right)
	\le \mathbb{E}_{ \overline{G}_M \sim \mathcal P_{\overline{G}_M} } \left[  \operatorname{TV}\left(\mathcal  P_{G_n \mid \overline{G}_M} \; , \mathcal Q_{ G_n \mid \overline{G}_M }\right) \right], \label{eq:TV_bounds}
\end{align}
where the first inequality follows from the data-processing inequality and the second one holds due to Jensen's inequality and the convexity of TV distance.

For notational simplicity, we abbreviate $\mathcal P_{G_n \mid \overline{G}_M=\overline{\operatorname{G}}_M}$ and $\mathcal Q_{G_n \mid \overline{G}_M=\overline{\operatorname{G}}_M}$ as $\mathtt P_{\overline{\operatorname{G}}_M}$
and $\mathtt Q_{\overline{\operatorname{G}}_M}$, respectively. Our main technical input is the next proposition.
\begin{proposition}\label{prop-variance}
	For any realization $\overline{\operatorname{G}}_M$ of $\overline{G}_M$, uniformly we have
	\[
	\mathbb{E}_{G\sim \mathtt P_{\overline{\operatorname{G}}_M}}\left(\frac{ \mathtt Q_{\overline{\operatorname{G}}_M}[G]}{\mathtt P_{\overline{\operatorname{G}}_M} [G] }-1\right)^2=O\left(\frac{1}{N}\right)\,.
	\]
\end{proposition}


\begin{proof}[Proof of Theorem~\ref{thm-main}]
	From Proposition~\ref{prop-variance} and Cauchy-Schwartz inequality, we get that for any realization $\overline{\operatorname{G}}_M$ of $\overline{G}_M,$
	\begin{align*}
		2 \operatorname{TV}(\mathtt P_{\overline{\operatorname{G}}_M},\mathtt Q_{\overline{\operatorname{G}}_M})=&\ \mathbb{E}_{G\sim \mathtt P_{\overline{\operatorname{G}}_M}}\left|\frac{\mathtt Q_{\overline{\operatorname{G}}_M}[G]}{\mathtt P_{\overline{\operatorname{G}}_M}[G]}-1\right|\\
        \le&\ \left(\mathbb{E}_{G\sim \mathtt P_{\overline{\operatorname{G}}_M}} \left(\frac{\mathtt Q_{\overline{\operatorname{G}}_M}[G]}{\mathtt P_{\overline{\operatorname{G}}_M}[G]}-1\right)^2\right)^{1/2} =O\left(\frac{1}{\sqrt{N}}\right)\,.\notag
	\end{align*}
	Combining the last display with~\eqref{eq:TV_bounds} and $N \gg \Delta^2$ yields
	the desired bound in \eqref{eq-goal} for $k=1$. The proof for $2\le k\le \Delta$ can be obtained by substituting $n$ with $n'\triangleq n-k+1$,  
	thereby concluding the proof of Theorem~\ref{thm-main}. 
\end{proof}


\section{Bounding the $\chi^2$-norm}\label{sec:bounding_norm}

We are left with proving Proposition~\ref{prop-variance}. Throughout this section, we fix a realization $\overline{\operatorname{G}}_M$ of network history $\overline{G}_M$ up to time $M$ and  simply write $\mathtt P_M,\mathtt Q_M$ for $\mathtt P_{\overline{\operatorname{G}}_M}, \mathtt{Q}_{\overline{\operatorname{G}}_M}$.  As promised earlier, the point of introducing the simplified models $\mathtt P_M,\mathtt Q_M$ is to achieve a simplification of the likelihood ratio $\dif\mathtt{Q}_M/\dif\mathtt{P}_M$, as we describe below.

To express the likelihood ratio explicitly, we introduce some notations. For $G$ sampled from either $\mathtt P_M$ or $\mathtt Q_M$,
denote $\mathscr{C}$ as the set of connected components of the subgraph of $G$ induced by vertices in $V_n \setminus V_M$ (those arrived during time  $ t \in [M+1,n]$). For each $\mathcal C\in \mathscr{C}$, we write its size $|\mathcal C|$ as the number of vertices in ${\mathcal C}$. For example, in Fig.~\ref{fig:PA_hard_component}, there are six components of size $1$, one component of size $2$,
and one component of size $3,$ highlighted by dashed ellipses.

For each component ${\mathcal C}\in \mathscr{C}$, we say an order $\prec$ on ${\mathcal C}$ is admissible, if for each $v\in {\mathcal C}$, $v$ connects 
to exactly $m$ vertices within $V_M \cup \{u\in {\mathcal C},u\prec v\}$. For each vertex $v \in V_n \setminus V_M$, we denote ${\mathcal C}(v)\in \mathscr{C}$ as the  component containing $v$ and define
$$
\lambda_v=\frac{\#\{\text{admissible orders on }{\mathcal C}(v)\text{ with $v$ being maximal}\}}{\#\{\text{admissible orders on }{\mathcal C}(v)\}}\,, 
$$
and $$
X_v=\prod_{w:w\sim v}\prod_{j=\mathsf{deg}_{G_{\setminus v}}(w)}^{\mathsf{deg}_{G}(w)-1}\frac{j+\delta'}{j+\delta} 
\,,
$$
where $G_{\setminus v}$ denotes the subgraph of $G$ with node $v$ and its incident edges deleted. As we explained in Section~\ref{sec:overview_proofs}, $\lambda_v$
is equal to the probability that the last-arriving vertex is $v$ conditional on it belongs to $\mathcal{C}(v)$; and 
$X_v$ represents the ratio of the attachment probability if there is a change from $\delta$ to $\delta'$ versus no change. 

Note that $\lambda_v\neq 0$ only if $\mathsf{deg}_G(v)=m$, 
and for $\mathsf{deg}_G(v)=m$ we have $X_v$ is uniformly bounded both from above and below. 
For example, in Fig.~\ref{fig:PA_hard_component}, consider the component of size 3. For the top vertex, we have $\lambda_v = 0$, while for each of the two bottom vertices, $\lambda_v = 1/2$ and $X_v = (1+\delta')/(1+\delta)$.

The next proposition gives the explicit expression of the likelihood ratio $\dif\mathtt{Q}_{M}/\dif\mathtt{P}_M$. Intuitively, as the identity of the last-arriving vertex is hidden, the expression involves a weighted average over the last $N$ arriving vertices $v \in V_n\setminus V_m$, where the weights are determined by the likelihood of $v$ being the last-arriving vertex.

\begin{proposition}\label{prop-likelihood-ratio-one-step}
	For any realization $G$ that is compatible with $\overline{\operatorname{G}}_M$, we have \begin{equation}\label{eq-likelihood-ratio}
		\frac{\mathtt Q_M[G]}{\mathtt P_M[G]}=\frac{C_1}{N}\sum_{v\in V_n \setminus V_{M}}|\mathcal C(v)| \lambda_{v}X_{v} \,,    
	\end{equation} 
	where $C_1=C_1(m,n,\delta,\delta')$ is a constant that is uniformly bounded in $n$.  
\end{proposition}

\begin{proof}
	Given $\overline{\operatorname{G}}_M$  and $G,$ let
	$
	\mathcal G = \left\{ \overline{G}_n: \overline{G}_M =   \overline{\operatorname{G}}_M, G_n = G \right\}
	$
	denote the set of network histories that are compatible with $\operatorname{G}_M$  and $G$. Then we have
	\[
	\mathtt P_{M}[G]=\sum_{ \overline{G}_n \in \mathcal{G} }\mathcal P_{M}[\overline{G}_n]\,,\quad \mathtt Q_{M}[G]=\sum_{ \overline{G}_n \in \mathcal{G} } \mathcal Q_{M}[\overline{G}_n]\,,
	\] 
	where $\mathcal P_M,\mathcal Q_M$ denote for $\mathcal P$ (distribution of $\overline{G}_n$under $\mathcal{H}_0$), $\mathcal Q$ (distribution of $\overline{G}_n$under $\mathcal{H}_1$) conditioning on $\overline{G}_M=\overline{\operatorname{G}}_M$. 
	By definition of the PA model as per~\eqref{eq-prob-def}, for any $ \overline{G}_n \in \mathcal{G}$,
	\[
	\mathcal P_M[\overline{G}_n]=\prod_{t\in [M+1,n]}\prod_{i=1}^m\frac{\mathsf{deg}_{G_{t,i}}(v_{t,i})-1+\delta}{(t-1)\delta+2m(t-2)+i-1}=C_0{\prod_{v\in V_n }\prod_{j=\mathsf{deg}_{G_M}(v)}^{\mathsf{deg}_G(v)-1}(j+\delta)}\,,
	\]
	where $v_{t,i}$ is the vertex $v_t$ attaching to in the graph $G_{t,i}$, and $$
    C_0=\prod_{t\in[M+1,n]}\prod_{i=1}^m{\big((t-1)\delta+2m(t-2)+i-1\big)^{-1}}.
    $$  
	Similarly, we have
	\begin{align*}
		\mathcal Q_M[\overline{G}_n] = & \prod_{t\in [M+1,n-1]}\prod_{i=1}^{m}\frac{\mathsf{deg}_{G_{t,i}}(v_{t,i})-1+\delta}{(t-1)\delta+2m(t-2)+i-1} \times\prod_{i=1}^m\frac{\mathsf{deg}_{G_{n,i}}(v_{n,i})-1+\delta'}{(n-1)\delta'+2m(n-2)+i-1}\\
		=&\ C_0{\prod_{v\in V_n}\prod_{j=\mathsf{deg}_{G_M}(v)}^{\mathsf{deg}_G(v)-1}(j+\delta)}\times C_1\prod_{u\sim v_n}\prod_{j\in \mathsf{deg}_{G_{n-1}}(u)}^{\mathsf{deg}_{G}(u)-1}\frac{j+\delta'}{j+\delta}\,,
	\end{align*}
	where $$
    C_1=\prod_{i=1}^m\frac{(n-1)\delta+2m(n-2)+i-1}{(n-1)\delta'+2m(n-2)+i-1}.
    $$
    Consequently,
	\[
	\frac{\mathtt Q_{M}[G]}{\mathtt P_M[G]}=C_1\cdot\frac{\sum_{ \overline{G}_n \in \mathcal{G} }X_{v_n}}{\sum_{ \overline{G}_n \in \mathcal{G} }1}=C_1\cdot \sum_{v \in V_n \setminus V_M } X_v \cdot \frac{| \{\overline{G}_n \in \mathcal G: v_n=v  \} | }{|\mathcal G|}\,.
	\]
	
	Recall that $\mathscr C$ is the set of  components of $G$ induced by vertices in $V_n \setminus V_M$. Therefore, denoting
	\begin{align*}
		\binom{N}{\mathscr{C}} \triangleq \binom{N}{(|\mathcal{C}'|)_{\mathcal{C}'\in \mathscr{C}}}\,,\quad \text{ and } \quad 
		\binom{N-1}{\mathscr{C},{\mathcal C}} \triangleq \binom{N-1}{(|\mathcal C'|)_{{\mathcal C'}\in \mathscr{C}\setminus\{{\mathcal C}\}},|\mathcal C|-1}\,, \forall {\mathcal C}\in \mathscr C\,,
	\end{align*}
	where $ \binom{n}{k_1, k_2, \ldots, k_m}=\frac{n!}{k_1!k_2!\cdots k_m!}$ is a multinomial coefficient, we have
	\[
	|\mathcal G|=\binom{N}{\mathscr C}\times\prod_{{\mathcal C}\in \mathscr C}\#\{\text{admissible orders on }{\mathcal C}\} \times (m!)^{N} \,,
	\]
	where the first term counts the ways to assign the $N$ vertices into the components; the second term counts the admissible orders of the vertex arrival times in each  component; and the last term counts the orders of attaching $m$ edges for each of the $N$ vertices.
	Similarly, for each vertex $v\in V_n \setminus V_{M}$, 
	\begin{align*}
		| \{\overline{G}_n \in \mathcal G: v_n=v \} |  =&\ \binom{N-1}{\mathscr C,{\mathcal C}(v)}\times \prod_{{\mathcal C}\in \mathscr C\setminus\{{\mathcal C}(v)\}}\#\{\text{admissible orders on }{\mathcal C}\}\\&\ \times \#\{\text{admissible orders on }{\mathcal C}(v)\text{ with $v$ being maximal}\} \times (m!)^{N}  \,,
	\end{align*}
	where the counting is similar as above except that since $v$ is the last vertex added, $v$ must be the last vertex in any admissible order over $\mathcal C(v).$
	
	As a result, 
	\[
	\frac{\mathtt Q_{M}[G]}{\mathtt P_M[G]}=\frac{C_1}{N}\sum_{v\in V_n \setminus V_{M}} |\mathcal C(v) | \lambda_{v}X_{v} \,,
	\]
	verifying \eqref{eq-likelihood-ratio}. Clearly from the expression of $C_1$, it is uniformly bounded in $n$, so the proof is completed.    
\end{proof}
For brevity, we denote 
\begin{equation}\label{eq-def-S}
	S = \frac{1}{N}\sum_{v\in V_n \setminus V_{M}} |\mathcal C (v)|\lambda_{v}X_{v} \,.
\end{equation}
Since the constant $C_1$ in \eqref{eq-likelihood-ratio} is uniformly bounded, we have
\[
\mathbb{E}_{G\sim \mathtt{P}_M}\left(\frac{\mathtt{Q}_M[G]}{\mathtt{P}_M[G]}-1\right)^2\le O(1)\times \operatorname{Var}_{\mathtt{P}_M}[S]\,.
\]

Therefore, it remains to show the following proposition.  
\begin{proposition}
	Let $\mathcal P'_M \equiv \mathcal P[\cdot\mid \overline{G}_M=\overline{\operatorname{G}}_M , \{v_t\}_{t=1}^n]$. It holds that 
	\begin{equation}\label{eq-final-goal}
		\operatorname{Var}_{\mathtt{P}_M}[S]=\operatorname{Var}_{\mathcal P'_M}[S]  
		=O\left(\frac{1}{N}\right)\,.
	\end{equation}
\end{proposition}

Note that in the first equality in \eqref{eq-final-goal}, we change the underlying probability measure from $\mathtt{P}_M$ to $\mathcal{P}'_M$. This is certainly doable as $S$ only depends on $G$ and $V_{M}$, and is independent of vertex arrival sequence $\{v_t\}_{t=1}^n$.

The remainder of this paper is dedicated to proving the last equality in \eqref{eq-final-goal}. Our main tool towards showing this is the Efron-Stein inequality (see e.g.~\cite[Theorem 3.1]{Boucheron13}), which states that for independent variables $Y_1,\ldots,Y_k$ and any function $f:\mathbb{R}^k\to \mathbb{R}$,
\[
\operatorname{Var}\big[f(Y_1,\dots,Y_k)\big]\le \sum_{i=1}^{k}\mathbb{E}\big[\big(f(Y_1,\dots,Y_i,\dots,Y_k)-f(Y_1,\dots,\widetilde{Y}_i,\dots,Y_k)\big)^2\big]\,,
\]
where $\widetilde{Y}_i$ is an independent copy of $Y_i$.  
To apply this inequality to control the variance of $S$, we first use some independent random variables to generate $\overline{G}_n \sim \mathcal P'_M$ and hence express $S$ as a function of them. While this is certainly doable using independent $\operatorname{Uni}([0,1])$ variables, what we will do is more delicate. The specific encoders we choose below serve for a combinatorial decoupling purpose, which will become clear in the proof of Lemma~\ref{lem-obs}.

To get a feeling of what we will do next, let us consider the special case of $m=1$ and $\delta=0$. We have the following alternative way to think about the preferential attachment model conditional on the vertex arrival sequence $\{v_t\}_{t=1}^n$: for $3\le t\le n$, conditional on $G_{t-1}$, we first transform each edge $(u,v)\in E_{t-1}$ to two directed edges $(u\rightarrow v)$ and $(v\rightarrow u)$, and then obtain $G_t$ by uniformly sampling a directed edge $\vec{e}$ 
and attaching $v_{t}$ to the starting point of $\vec{e}$. Since the size of the directed edge set of $G_{t-1}$ is always $2(t-2)$, the graph $G_n$ can be encoded by $n-2$ independent random variables that have uniform distributions on $[2],[4],\dots,[2(n-2)]$ (recalling $[k]$ is the set $\{1,\dots,k\}$). 
For the general cases $m\ge 1$ and $\delta>-m$, the sampling procedure in \eqref{eq-prob-def} can be encoded by a mixture of the uniform distribution on vertices and the uniform distribution on directed edges (or a subset of directed edges if $\delta<0$), as we elaborate below.

For ease of presentation, define the sets
\[
\mathcal I\triangleq \{(t,i):M+1\le t\le n,1\le i\le m\}\subset \mathcal I_0\triangleq \{(t,i):3\le t\le n,1\le i\le m\}\,,
\]
and let $\prec$ be the lexicographical order of $\mathcal I_0$. Further, we denote $\kappa\equiv \lceil \max\{-\delta, 0\}\rceil\in \{0,\dots,m\}$. Recall that for $(t,i)\in \mathcal I_0$, $v_{t,i}$ is the vertex $v_t$ attaching to in the graph $G_{t,i}$. For any $(t,i)\in \mathcal I$, we define two multi-sets of directed edges $\vec{E}^{\uparrow}_{t,i},\vec{E}^{\downarrow}_{t,i}$ as (below we use $k\times A$ to denote the multi-set of $k$ copies of $A$): 
\begin{align*}
	\vec{E}^{\uparrow}_{t,i} & \triangleq \big((m-\kappa)\times \{(v_1\rightarrow v_2),(v_2\rightarrow v_1)\}\big)\cup \{(v_{t'}\rightarrow v_{t',i'}):3\le t'\le t-1,1\le i'\le m-\kappa\}\,,\\
	\vec{E}^{\downarrow}_{t,i} & \triangleq \{(v_{t',i'}\rightarrow v_{t'}):(t',i')\in \mathcal I_0,(t',i')\prec (t,i)\}\,.
\end{align*}
We also let the multi-set $\vec{E}_{t,i}=\vec{E}_{t,i}^{\uparrow}\cup \vec{E}_{t,i}^{\downarrow}$. Clearly $|\vec{E}_{t,i}|=(t-1)(m-\kappa)+(t-3)m+i-1\stackrel{\Delta}{=}K_{t,i}$, and we label the edges in $\vec{E}_{t,i}$ by $\vec{e}_1,\dots.\vec{e}_{K_{t,i}}$ in an arbitrary deterministic way.

Consider independent random variables $\{U_{t,i}\}_{(t,i)\in \mathcal I}$,
such that for each $(t,i)\in \mathcal I$, $U_{t,i}$ is a random triple $$\{I_{t,i},W_{t,i},Y_{t,i}\}\in \{0,1\}\times [t-1]\times [K_{t,i}]$$
that distributes as the product distribution
\begin{equation}\label{eq-distribution-of-U}
	\operatorname{Ber}\left(\frac{(t-1)(\delta+\kappa)}{(t-1)(\delta+\kappa)+2m(t-2)+i-1}\right)\otimes \operatorname{Uni}([t-1])\otimes\operatorname{Uni}([K_{t,i}])\,,
\end{equation}
where $\operatorname{Ber}(p)$ denotes the Bernoulli distribution with parameter $p$ and $\operatorname{Uni}(A)$ denotes the uniform measure on set $A$.

We now use the random variables $\{U_{t,i}\}_{(t,i)\in \mathcal I}$ to generate $\overline{G}_n \sim \mathcal P'_M$. First, we note that for $(t,i)\in \mathcal I_0\setminus \mathcal I$, $v_{t,i}$ is deterministic given $\overline{G}_M=\overline{\operatorname{G}}_M$. We then sequentially generate $v_{t,i}$ for $(t,i)\in \mathcal I$ according to the order $\prec$. For each $(t,i)\in \mathcal I$, let $(t',i')$ be the predecessor of $(t,i)$ in $\mathcal I$ (the minimal element has predecessor $\emptyset$), and assume we have already generated the graph $G_{t',i'}$ (we use the convention that $G_\emptyset=\operatorname{G}_M$). To obtain the graph $G_{t,i}$, we connect an edge from $v_t$ to a random vertex $v_{t,i}\in V_{t-1}$
determined by the following rule:
\begin{itemize}
	\item If $I_{t,i}=1$, we let $v_{t,i}=v_{W_{t,i}}$;
	\item If $I_{t,i}=0$ and $\vec{e}_{Y_{t,i}}=(v\rightarrow v')\in \vec{E}_{t,i}$, we let $v_{t,i}=v$. 
\end{itemize}

It is straightforward to check that for $U_{t,i}=\{I_{t,i},W_{t,i},Y_{t,i}\}$ sampled from the distribution in \eqref{eq-distribution-of-U}, $G_{t,i}$ has the same distribution as in Definition~\ref{def-preferential-attachment-model}. Specifically, for any vertex $v \in V_{t-1}$,
\begin{align*}
\mathbb{P} \left[ v_{t,i}=v \mid G_{t,i-1} \right] 
&= \frac{(t-1)(\delta+\kappa)}{(t-1)(\delta+\kappa)+2m(t-2)+i-1} \times \frac{1}{t-1} \\
& \quad + \frac{2m(t-2)+i-1}{(t-1)(\delta+\kappa)+2m(t-2)+i-1} \times \frac{\deg_{G_{t,i}}(v)-\kappa}{K_{t,i}} \\
& = \frac{\deg_{G_{t,i}}(v)+\delta}{(t-1)\delta+ 2m(t-2)+i-1},
\end{align*}
where the first equality holds because the number of directed edges in $\vec{E}_{t,i}$ that start from  $v$ is equal to $\deg_{G_{t,i}}(v)-\kappa$; the second equality holds by plugging in $K_{t,i}=2m(t-2)-(t-1)\kappa+i-1$
and $\kappa=\max\{-\delta,0\}.$

Using this encoding procedure, we can express $S$ defined in \eqref{eq-def-S} by some function $f_S$ of $\{U_{t,i}\}_{(t,i)\in \mathcal I}$. Applying Efron-Stein inequality, we obtain
\[
\operatorname{Var}_{\mathcal P'_M}[S]\le \sum_{(t,i)\in \mathcal I}\mathbb{E}\big[\big(f_S(U_{M+1,1},\dots,U_{t,i},\dots,U_{n,m}\big)-f_S(U_{M+1,1},\dots,\widetilde{U}_{t,i},\dots,U_{n,m})\big)^2\big]\,,
\]
where $\widetilde{U}_{t,i}$ is an independent copy of $U_{t,i}$. 

Now fix an arbitrary pair $(t,i)\in \mathcal I$. We claim that
\begin{equation}\label{eq-variance-for-each-tj}
	\mathbb{E}\big[\big(f_S(U_{M+1,1},\dots,U_{t,i},\dots,U_{n,m}\big)-f_S(U_{M+1,1},\dots,\widetilde{U}_{t,i},\dots,U_{n,m})\big)^2\big]=O\left(\frac{1}{N^2}\right)\,.
\end{equation}
Provided that this is true, since $|\mathcal I|=mN=O(N)$, we get the desired bound $\operatorname{Var}_{\mathcal P'_M}[S]=O\big(\frac 1 N\big)$.  

For any realization of $U_{M+1,1},\dots,U_{n,m}$ and $\widetilde{U}_{t,i}$, with a slight abuse of notation, we denote 
\[
f_S=f_S(U_{M+1,1},\dots,U_{t,i},\dots,U_{n,m}\big), \quad \widetilde{f}_S =f_S(U_{M+1,1},\dots,\widetilde{U}_{t,i},\dots,U_{n,m})\,.
\]
Moreover, we let $\mathscr{C},\widetilde{\mathscr{C}}$ be the set of components induced by the vertices in $V_n \setminus V_M$ in the graphs $G,\widetilde{G}$ generated from $U_{M+1,1},\dots,U_{t,i},\dots, U_{n,m}$ and $U_{M+1,1},\dots,\widetilde{U}_{t,i},\dots,U_{n,m}$, respectively. We also define ${\mathcal C}(v)$, $\widetilde{{\mathcal C}}(v)$, $ \lambda_v$, $ \widetilde{\lambda}_v$, $ X_v$, and $ \widetilde{X}_v$ correspondingly as before. To bound the left-hand side of \eqref{eq-variance-for-each-tj}, we use the following crucial observation.

\begin{lemma}\label{lem-obs}
	For any realization of $U_{M+1,1},\dots,U_{n,m}$ and $\widetilde{U}_{t,i}$, deterministically it holds
	\[
	|f_S-\widetilde{f}_S|\le O\left(\frac{ |\mathcal C(v_t)|+ |\widetilde{{\mathcal C}}(v_t)|}{N}\right)\,.
	\]
\end{lemma}
\begin{proof}
	Note that resampling even a single variable, $\widetilde{U}_{t,i}$, could trigger a cascade of differences between $G$ and $\widetilde{G}$, ultimately resulting in a significant discrepancy between $f_S$ and $\widetilde{f}_S.$ To demonstrate that this cannot occur, our proof relies critically on our structured rule used to generate the graph  $G \sim \mathcal{P}'_M$ via the independent random variables $\{U_{t,i}\}_{(t,i)\in \mathcal{I}}$. 

	Recall that as in \eqref{eq-def-S},
	\[
	f_S=\frac{\sum_{v \in V_n \setminus V_{M}}\lambda_v X_v |\mathcal C (v)|}{N},\quad \widetilde{f}_S=\frac{\sum_{ v \in V_n \setminus V_{M}}\widetilde{\lambda}_v\widetilde{X}_v |\widetilde{{\mathcal C}}(v)|}{N}\,.
	\]
	For a vertex $v \in V_n \setminus V_M$, if neither $v\in {\mathcal C}(v_t)$ nor $v\in \widetilde{{\mathcal C}}(v_t)$, then ${\mathcal C}(v)$ and $\widetilde{{\mathcal C}}(v)$ are identical (thus also $\lambda_v=\widetilde{\lambda}_v,X_v=\widetilde{X}_v$). This is because by our rule of generating the graph $G$ and $\widetilde{G}$, the difference between $U_{t,i}$ and $\widetilde{U}_{t,i}$ can only affect the graph structure of the component containing $v_t$. Therefore, the terms in the sum regarding  vertices $v\notin {\mathcal C}(v_t)\cup \widetilde{{\mathcal C}}(v_t)$ cancel out in the difference $f_S-\widetilde{f}_S$, yielding that
	\begin{align}
		|f_S-\widetilde{f}_S|=&\ \frac{1}{N}\Bigg|\sum_{v\in {\mathcal C}(v_t)\cup \widetilde{{\mathcal C}}(v_t)}\left(\lambda_vX_v |{\mathcal C}(v)|-\widetilde{\lambda}_v\widetilde{X}_v |\widetilde{{\mathcal C}}(v)|\right)\Bigg|\nonumber\\
		\le&\ \frac{1}{N}\Bigg|\sum_{v\in {\mathcal C}(v_t)\cup \widetilde{{\mathcal C}}(v_t)}\lambda_v X_v |{\mathcal C}(v)|\Bigg|+\frac{1}{N}\Bigg|\sum_{v\in {\mathcal C}(v_t)\cup \widetilde{{\mathcal C}}(v_t)}\widetilde{\lambda}_v\widetilde{X}_v |\widetilde{{\mathcal C}}(v)|\Bigg|\,.\label{eq-two-terms}
	\end{align}
	
	Now we claim that for any $v\in \widetilde{{\mathcal C}}(v_t)\setminus {\mathcal C}(v_t)$, it holds that ${{\mathcal C}}(v)\subset \widetilde{{\mathcal C}}(v)=\widetilde{{\mathcal C}}(v_t)$. 
	To see this, consider any edge $e = (v_{t'}, v_{t',i'}) \in {{\mathcal C}}(v)$. Since $v\notin {{\mathcal C}}(v_t)$, $t' \neq t.$ 
    \begin{itemize}
    \item First, suppose $t'<t.$ In this case, $v_{t',i'}$ must coincide with $\widetilde{v}_{t',i'}$, the vertex to which $v_{t'}$ attached in $\widetilde{G}$. Therefore, $e$ also appears in $\widetilde{G}$. 
	\item Second, suppose instead $t'>t.$ By our generating rule, $v_{t',i'}$ can only differ from $\widetilde{v}_{t',i'}$, if 
	$v_{t',i'}$ is chosen as an endpoint of an edge in $G \setminus \widetilde{G}$. However, $G \setminus \widetilde{G}$ must be contained within $ \mathcal C(v_t)$. It follows that $v_{t',i'}$
	and consequently $e$ must be contained in $\mathcal C(v_t)$. 
	Since $e \in \mathcal C(v)$, this would imply $v \in \mathcal C(v_t)$, which contradicts our assumption that $v \notin \mathcal C(v_t)$. Thus we conclude that  
	$v_{t',i'}=\widetilde{v}_{t',i'}$, and therefore $e$ also appears in $\widetilde{G}$.
    \end{itemize}
    In conclusion, we have shown that $\mathcal C(v) \subset \widetilde{G}$, and hence ${{\mathcal C}}(v)\subset \widetilde{{\mathcal C}}(v)=\widetilde{{\mathcal C}}(v_t)$.





	It follows that $\widetilde{{\mathcal C}}(v_t)\setminus {{\mathcal C}}(v_t)$ can be partitioned into a disjoint union of components ${{\mathcal C}}_1,\dots,{{\mathcal C}}_L\in \mathscr{C}$ satisfying 
	\begin{equation}\label{eq-sum-Cl-bound}
		|\mathcal C_1| +\cdots+ |\mathcal C_L|\le|\widetilde{{\mathcal C}}(v_t)\setminus {\mathcal C}(v_t)|\le |\widetilde{{\mathcal C}}(v_t)| \,.
	\end{equation}
	Consequently, we have
	\begin{align*}
		\Bigg|\sum_{v\in {\mathcal C}(v_t)\cup \widetilde{{\mathcal C}}(v_t)}\lambda_vX_v |\mathcal C(v)| \Bigg|=&\ \Bigg| |\mathcal C(v_t)| \sum_{v\in {\mathcal C}(v_t)}\lambda_vX_v+\sum_{l=1}^L |\mathcal C_l | \sum_{v\in {\mathcal C}_l}\lambda_v X_v\Bigg|\\
        =&\ O\left( |{\mathcal C(v_t)}|+ |\widetilde{{\mathcal C}}(v_t)|\right)\,.
	\end{align*}
	Here in the last inequality, we used \eqref{eq-sum-Cl-bound} and the fact that for any ${\mathcal C}\in \mathscr C$, $\sum_{v\in {\mathcal C}}\lambda_v=1$ and $X_v$ is uniformly bounded. Hence the first term in \eqref{eq-two-terms} is $O\big((| {\mathcal C(v_t)}|+|\widetilde{\mathcal C}(v_t)|)/{N}\big)$. Similarly one can bound the second term in \eqref{eq-two-terms}, and the result follows.
\end{proof}

In light of Lemma~\ref{lem-obs}, in order to prove \eqref{eq-variance-for-each-tj}, we only need to show that
\[
\mathbb{E}\Big[( | {\mathcal C}(v_t)| + |\widetilde{{\mathcal C}}{(v_t)}|)^2\Big]\le \mathbb{E}\Big[2 |{\mathcal C}(v_t)|^2+2|\widetilde{{\mathcal C}}(v_t)|^2\Big]=4\mathbb{E}\big[|{\mathcal C}(v_t)|^2\big]=O(1)\,.
\]
This is done in the next proposition.

\begin{proposition}\label{prop-EnC(v)2=O(1)}
	For any vertex $v \in V_n \setminus V_M$, it holds uniformly that $\mathbb{E}_{\mathcal P'_M}\big[|{\mathcal C}(v)|^2\big]=O(1)$. 
\end{proposition}


Fixing a vertex $v \in V_n \setminus V_M$, we will show that $\mathcal{P}'_M[ |{\mathcal C(v)}| \ge k]$ decays exponentially in $k$, which is certainly enough to conclude Proposition~\ref{prop-EnC(v)2=O(1)}. To do this, we will prove that the size of ${\mathcal C}(v)$ is stochastically dominated by the size of a sub-critical branching tree, as we define next. 

\begin{definition}
	\label{def-dominance-tree}
	We define $\mathcal T$ as the random tree generated from the branching process with offspring distribution $X$, where $X$ satisfies
	\begin{align*}
		X=Y+Z,\quad Y\sim \operatorname{Binom}(m,2N/n),\quad Z\sim \operatorname{Geo}(1-(2m N)/n),\quad Y,Z\text{ are independent}\,.
	\end{align*}
	Here, $\operatorname{Binom}(n,p)$ 
	denotes the binomial distribution with parameters $n,p$, and $\operatorname{Geo}(q)$ denotes the geometric distribution with success probability $q$ (the distribution of the number of failures before the first success in independent repeated trials). 
\end{definition}

Consider the breadth-first-search (BFS) process that constructs a spanning tree of ${\mathcal C}(v)$. Starting with $\operatorname{T}_0$ as a singleton $\{v\}$, we construct a sequence of trees $\operatorname{T}_1,\operatorname{T}_2,\dots$ as follows. For each $t\in \mathbb{N}$, we denote the leaves of $\operatorname{T}_{t-1}$ as $v_1,\dots,v_{\ell_t}$. Sequentially for $1\le l\le \ell_t$, let $\operatorname{N}_l^t$ be the set of vertices in $V_n\setminus \big(V_M \cup V(\operatorname{T}_{t-1})\cup \operatorname{N}_1^t\cup\cdots\cup \operatorname{N}_{l-1}^t\big)$ that connect to $v_l$ in $G$. We then connect each vertex in $\operatorname{N}_l^t$ to $v_l,1\le l\le \ell_t$ to get the tree $\operatorname{T}_t$ (so the leaf set of $\operatorname{T}_{t}$ is $\operatorname{N}_1^t\cup\cdots\cup\operatorname{N}_{\ell_t}^t$).  It is clear that the vertex set of $V(\operatorname{T}_\infty)$
equals the vertex set of ${\mathcal C}(v)$. We now show the BFS-tree $\operatorname{T}_\infty$ of ${\mathcal C}(v)$ is stochastically dominated by $\mathcal T$.

\begin{lemma}\label{lmm:sd}
	There is a coupling of $G\sim \mathcal P'_M$ and $\mathcal T$, such that almost surely, $\operatorname{T}_\infty$ is a subtree of $\mathcal T$. As a result, we have for any $k\ge1$, $\mathcal P'_M[|{\mathcal C}(v)|\ge k]\le \mathbb{P}[|\mathcal T|\ge k]$.  
\end{lemma}
\begin{proof}
	It suffices to show the following: for any realization of $\mathtt{T}_{t-1}$ with leaf set $\mathtt{L}_{t-1}=\{v_1,\dots,v_{\ell_t}\}$, any $1\le l\le \ell_t$, and any realization of $\operatorname{N}_1^t,\dots,\operatorname{N}_{l-1}^t$, the conditional distribution of $|\operatorname{N}_l^t|$ given these realizations is stochastically dominated by $X$. 
	

	To see this, we first note that the vertices in $\operatorname{N}_l^t$ are of two types: (i) with arrival times smaller than $v_l$; (ii) with arrival times larger than $v_l$. Vertices of type (i) must be among the $m$ vertices that $v_l$ attaches to when it arrives, which we refer to as the $v_l$-attaching vertices. Given the realization of $\operatorname{T}_{t-1}$ and $\operatorname{N}_1^t,\dots,\operatorname{N}_{l-1}^t$, we may have already revealed some of the $v_l$-attaching vertices.
	However, the remaining of them are still independently sampled from those vertices that arrive earlier than $v_l$ with probability proportional to current degrees plus the additive term $\delta$, while conditioning on that they do not belong to $V(\operatorname{T}_{t-1})\cup\operatorname{N}_1^t\cup\cdots\cup\operatorname{N}_{l-1}^t$. Nevertheless, each one of the remaining $v_l$-attaching vertices is sampled from $V_M$ with probability at least $\frac{2m(M-1) +\delta M}{2m (n-1) +\delta n} \ge \frac{M-2}{n-2}$, where in the denominator we overcount the total degrees plus additive $\delta$ terms.  Therefore, 
	conditioning on the realization of $\operatorname{T}_{t-1}$ and $\operatorname{N}_1^t,\dots,\operatorname{N}_{l-1}^t$, a $v_l$-attaching vertex belongs to $V_n\setminus \big(V_M \cup V(\operatorname{T}_{t-1})\cup \operatorname{N}_1^t\cup\cdots\cup \operatorname{N}_{l-1}^t\big)$ with probability  at most $1-\frac{M-2}{n-2}=\frac{N}{n-2} \le \frac{2N}{n}$ for $n \ge 4$ and hence
	\[
	\#\{v\in \operatorname{N}_l^t\text{ of type (i)}\}\preceq_{\operatorname{SD}}\operatorname{Binom}(m,2N/n)\,.
	\]
	
	For vertices of type (ii), they must attach to $v_l$ when they arrive. Note that for any $j\ge 0$, when $\mathsf{deg}(v_l)=m+j$, $v_l$ is attached by a new vertex at any specific time conditioning on all the realizations beforehand is at most $\frac{m+j+\delta}{2m(M-1)+\delta M} \le \frac{j+1}{M}$ for $M\ge 2$, where in the denominator we undercount the total degrees plus additive $\delta$ terms by only including the contributions from $V_M$. 
	It follows from a union bound that conditioning on the realization of $\operatorname{T}_{t-1}$ and $\operatorname{N}_1^t,\dots,\operatorname{N}_{l-1}^t$,  the probability of $\{\#\{v\in \operatorname{N}_l^t\text{ of type (ii)}\}\ge k\}$ for any $k\ge 0$ is at most
	\begin{equation}\label{eq-type-1-SD}
		\binom{mN}{k}\cdot \prod_{j=0}^{k-1}\frac{j+1}{M} \le \left(\frac{m N}{M}\right)^k \le \left(\frac{2m N}{n}\right)^k,
	\end{equation}
	Here, the binomial term accounts for the choices of the first $k$ times that $v_l$ is attached by later vertices. 
	In addition, the last inequality holds due to $M=n-N=(1-o(1))n \ge n/2.$
	It follows that 
	\begin{equation}\label{eq-type-2-SD}
		\#\{v\in \operatorname{N}_l^t\text{ of type (ii)}\}\preceq_{\operatorname{SD}}\operatorname{Geo}(1-(2mN)/n)\,.
	\end{equation}
	Additionally, the numbers of type (i) and type (ii) vertices in $\operatorname{N}_l^t$ are conditionally independent. Combining \eqref{eq-type-1-SD}, \eqref{eq-type-2-SD} yields the desired result.
\end{proof}

It is well-known that the size of a sub-critical branching tree with exponential tail off-spring distribution decays exponentially. For completeness we present a proof here.

\begin{lemma}\label{lmm:bt}
	It holds that for any $k\ge 1$, $\mathbb{P}[|\mathcal T|\ge k]\le 2e^{-k+1}$.
\end{lemma}
\begin{proof}
	Write $F(s)=\sum_{k=0}^\infty \mathbb{P}[X=k]s^k$ as the moment generating function of $X$.     Since $X=Y+Z$ where $Y\sim \operatorname{Binom}(m,2N/n)$ and $ Z\sim \operatorname{Geo}(1-(2mN)/n)$ are independent, we have
	\[
	F(2e)=\mathbb{E}\big[(2e)^X\big]=\mathbb{E}\big[(2e)^Y\big]\cdot \mathbb{E}\big[(2e)^Z\big]=\left(1+\frac{2(2e-1)N}{n}\right)^m\cdot \frac{1-2mN/n}{1-4em N/n}\,,
	\]
	which is $1+o(1)$ as $N=o(n)$ and $m=O(1)$. Clearly $F$ is increasing, so we conclude that $F$ converges on $[1,2e]$. Moreover, as $eF(2e)=e+o(1)<2e$ and $eF(1)=e>1$, there must be a fixed point of $eF(x)$ that lies in $[1,2e]$. 
	
	We then define $\phi(t)=\mathbb{E}e^{|\mathcal T_t|}$ for $t\ge 0$. Let $v_1,\dots,v_X$ be the children of the root, and let $\mathcal T_{t-1}^k,1\le k\le X$ be the subtree of $\mathcal T_{t}$ rooted at $v_k$. It is clear that $\mathcal T_{t-1}^k$ are independent and have the same distribution as $\mathcal T_{t-1}$, thus $\phi(\cdot)$ satisfies the following recursive relation:
	\begin{equation}\label{eq-recursion}
		\phi(t)=\mathbb{E}e^{1+|\mathcal T|_{t-1}^1+\dots+|\mathcal T_{t-1}^X|}=eF(\phi(t-1))\,.
	\end{equation}
	Note that $\phi(0)=1$, it follows from \eqref{eq-recursion} that $\phi(t)$ is finite for any $t$, and as $t\to\infty$, $\phi(t)$ converges to the smallest fixed point of $eF(x)$ on $[1,2e]$. In conclusion, we get 
	\[
	\mathbb{E}[e^{|\mathcal T|}]=\lim_{t\to\infty}\phi(t)\le 2e\,,
	\]
	where the first equality follows from the monotone convergence theorem. Applying Markov's inequality, we get the desired result.
\end{proof}

Finally we prove Proposition~\ref{prop-EnC(v)2=O(1)} and concludes this paper.

\begin{proof}[Proof of Proposition~\ref{prop-EnC(v)2=O(1)}]
	By Lemmas~\ref{lmm:sd} and~\ref{lmm:bt}, we have
	\begin{equation*}
		\mathcal P'_M\left[ |{\mathcal C}(v)|\geq k\right]\leq \mathbb{P}[|\mathcal T|\ge k] \leq 2e^{-k+1}\,.
	\end{equation*}
	This implies
	\begin{equation}
		\begin{split}
			&\mathbb{E}_{\mathcal P'_M}\left[ |{\mathcal C}(v)|^2 \right]=\sum_{k=1}^\infty  (2k-1)\mathcal P'_M\left[|{\mathcal C}(v)| \geq k\right]
			\leq \sum_{k=1}^\infty 
			4ke^{-k+1}=O(1)\,.
		\end{split}\notag
	\end{equation}
\end{proof}


\section{Discussion and outlook}
\label{sec:discussion}

In this paper,
we prove the changepoint detection threshold is $\tau_n=n-o(\sqrt{n})$, confirming a conjecture of~\cite{bet2023detecting}.
Along the way, we show the changepoint localization threshold is also $\tau_n=n-o(\sqrt{n})$, matching the upper bound in~\cite{bhamidi2018change}.   The key components of our proof  include reducing the general problem to bounding the TV distance when the changepoint occurs at $n-1$ via interpolation, revealing network history up to time $n-o(n)$ to simplify the likelihood ratio, and bounding the second-moment of the likelihood ratio through the  Efron-Stein inequality and coupling arguments.

Our study opens several interesting directions for future research on changepoint detection in random graphs. 
For instance, a natural extension is to explore more general attachment rules, where at time $t$, the probability of connecting $v_t$ to $u$ given $G_{t-1}$ is proportional to $f(\mathsf{deg}_{G_{t-1}}(u))$. 
 In this work, we assume an affine function $f$ and detect changes in the additive shift $\delta$. A close inspection of our proofs reveals that our general proof strategy does not depend on the specific functional form, except for the encoding scheme of the PA model used in the application of the Efron-Stein inequality. Future work could investigate changes in more general functional forms, such as polynomial functions (e.g., changes in degree), or a transition from one arbitrary function $f_1(\cdot)$ to another $f_2(\cdot)$—introducing a structural change in the network evolution process. Additionally, existing research on changepoint detection and location primarily concentrated on models with fixed initial degrees  $m$; it would be interesting to extend the analysis to models with random initial degrees $m$~\cite{deijfen2009preferential}. 
 More broadly, our framework can also be extended to settings where the changepoint occurs only transiently or multiple times~\cite{cirkovic2022likelihood}, and alternative dynamic network models, such as dynamic stochastic block models or dynamic graphon models~\cite{wang2021optimal,enikeeva2025change}.

\section*{Acknowledgements}
The authors are grateful for inspiring discussions on the preferential attachment models with Simiao Jiao, Weijia Li, and Xiaochun Niu. J.~Xu also thanks the organizers of the conference ``Statistical and Probabilistic Analysis of Random Networks and Processes'', where he first learned about the changepoint detection threshold conjecture.

H. Du is partially supported by an NSF-Simons research collaboration grant (award number 2031883). S. ~Gong is partially supported by National Key R\&D program of China (No. 2023YFA1010103) and NSFC Key Program (Project No. 12231002). Part of the work is carried out while S.~Gong is in residence at the Mathematical Sciences Research Institute in Berkeley, California, during the Spring 2025 semester.
J.~Xu is supported in part by an NSF CAREER award CCF-2144593.

	\bibliographystyle{plain} 
	\bibliography{References.bib,network_archaeology_refs.bib}       

	
\end{document}